\documentclass[
	paper=A4,
	pagesize=auto,
	fontsize=12pt
]{scrartcl}
\usepackage{libertine}
\recalctypearea

\usepackage[ngerman,british]{babel}
\usepackage[style=alphabetic,url=false,isbn=false,doi=true,eprint=false,backend=biber]{biblatex}
\usepackage{graphicx}
\usepackage{amsmath,amssymb,amsthm}
\usepackage{csquotes}
\usepackage{bm}
\usepackage{enumerate}

\usepackage{hyperref}
\hypersetup{
  colorlinks
}
\usepackage{cleveref}

\theoremstyle{plain}
\newtheorem{theorem}{Theorem}[section]
\newtheorem{lemma}[theorem]{Lemma}

\theoremstyle{remark}
\newtheorem{example}[theorem]{Example}
\newtheorem{remark}[theorem]{Remark}

\theoremstyle{definition}
\newtheorem{definition}[theorem]{Definition}

\newlength{\JZHeightOfX}
\newcommand{\JZOrcidlink}[1]{
\setlength{\JZHeightOfX}{\fontcharht\font`X}
\includegraphics[height=\JZHeightOfX]{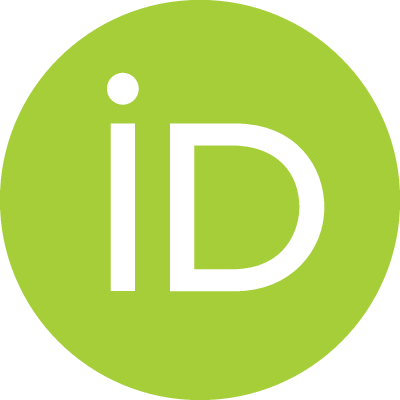}
\href{https://orcid.org/#1}{#1}
}

\begin{document}
\author{
Jan Fischer\\
\small{Faculty of Mathematics and Computer Science, Friedrich-Schiller-University, Jena, Germany}\\
Jobst Ziebell\hspace{2em}\JZOrcidlink{0000-0002-9715-6356}\\
\small{Faculty of Mathematics and Computer Science, Friedrich-Schiller-University, Jena, Germany}
}
\title{Tikhonov Well-Posedness and Differentiability on Asymmetrically Normed Spaces}
\date{\today}
\maketitle

\begin{abstract}
On normed vector spaces there is a well-known connection between the Tikhonov well-posedness of a minimisation problem and the differentiability of an associated convex conjugate function.
We show how this duality naturally generalises to the setting of asymmetrically normed spaces and prove a universal differentiability property of the convex conjugate of the cumulant-generating function of a mean-zero measure on a locally convex space.
\end{abstract}
\section{Introduction}
Given a normed space $X$ and a function $g : X \to (-\infty, \infty]$, the minimisation problem $(g,X)$ is Tikhonov well-posed if $g$ attains its infimum at some point $x \in X$ and, additionally, every minimising sequence of $g$ converges to $x$.
In \cite[Theorem 3.9.1]{src:Zalinescu:ConvexAnalysisInGeneralVectorSpaces} it is shown that this property is intimately related to differentiability properties of an associated convex conjugate function (also referred to as Fenchel conjugate or Legendre transform).
Closely related results are found in \cite{src:Soloviov,src:AsplundRockafellar}.

This property can also be studied on asymetrically normed spaces, which requires the extension of the convex conjugate operation to asymmetrically normed spaces and normed cones respectively.
The authors are unaware of prior work in this direction.
There is however a recent study on nonlinear Fenchel conjugates \cite{src:NonlinearFenchelConjugates} extending some ideas and properties of convex conjugation to functions on arbitrary sets.
\section{Preliminaries}
\subsection{Cones and Asymmetrically Normed Spaces}
As in \cite{src:Valero:QuotientNormedSpace} we define a \textbf{cone} as a set $X$ together with two operations $(+,\cdot)$ such that $(X,+)$ is an abelian monoid and $\cdot : \mathbb{R}_{\ge 0} \times X \to X$ such that
\begin{enumerate}[(i)]
\item $r \cdot \left( s \cdot x \right) = \left( r s \right) \cdot x$
\item $r \cdot \left( x + y \right) = r \cdot x + r \cdot y$
\item $\left( r + s \right) \cdot x = r \cdot x + s \cdot x$
\item $1 \cdot x = x$
\item $0 \cdot x = x$
\end{enumerate}
for all $r, s \ge 0$ and $x, y \in X$.
In particular, we shall consider a \textbf{normed cone} as defined in \cite{src:Valero:QuotientNormedSpace}, i.e. as a pair $(X, p)$ of a cone and a function $p : X \to \mathbb{R}_{\ge 0}$ satisfying
\begin{enumerate}[(i)]
\item $p \left( x \right) = 0 \iff x = 0$
\item $p \left( r x \right) = r p \left( x \right)$
\item $p \left( x + y \right) \le p \left( x \right) + p \left( y \right)$
\end{enumerate}
for all $r \ge 0$ and all $x, y \in X$.
We also follow \cite{src:Cobzaş:FunctionalAnalysisinAsymmetricNormedSpaces} and define an \textbf{asymmetric norm} on a linear space $X$ as a function $p : X \to \mathbb{R}_{\ge 0}$ satisfying
\begin{enumerate}[(i)]
	\item $x = 0 \iff p \left( x \right) = p \left( -x \right) = 0$
	\item $p \left( r x \right) = r p \left( x \right)$
	\item $p \left( x + y \right) \le p \left( x \right) + p \left( y \right)$
\end{enumerate}
for all $r \ge 0$ and all $x, y \in X$.
We will then refer to the pair $(X,p)$ as an \textbf{asymmetrically normed space}.
Note that an asymmetrically normed space is in general not a normed cone.

In both the linear and the conic setting a neighbourhood basis of the induced topology at a point $x \in X$ is simply given by the collection of sets of the form $B_r(x) = x + r p^{-1}([0,1))$ for all $r > 0$.
Clearly, $X$ is first-countable.
Given an asymmetrically normed space $(X, p)$, we may also talk about the \textbf{conjugate} space $\bar{X}$ which is given as the corresponding tuple $(X, \bar{p})$ where $\bar{p}(x) = p(-x)$ for all $x \in X$.
Clearly, $\bar{X}$ is also an asymmetrically normed space.
It is easy to see that $\max\{p, \bar{p}\}$ is an ordinary norm on the space $X$ and we shall refer to the corresponding normed space as $X_s$ (\enquote{s} for \emph{symmetric}).

As is common for normed spaces, we shall also write $\Vert \cdot \Vert$ for the norm of a normed cone or an asymmetrically normed space.
Given an asymmetrically normed space $(X, \Vert \cdot \Vert_X)$ we shall also define its \textbf{dual cone} $X^*$ as the space of all real-valued upper semicontinuous linear functionals on $X$.
We equip $X^*$ with the functional $\Vert \cdot \Vert_{X^*} : X^* \to \mathbb{R}_{\ge 0}$ given by
\begin{equation}
\phi \mapsto \sup_{x \in B_1(0)} \phi \left( x \right)
\end{equation}
for all $\phi \in X^*$.
This turns $X^*$ into a normed cone since $B_1(0)$ is absorbing in $X$.
It is shown in \cite{src:Cobzaş:FunctionalAnalysisinAsymmetricNormedSpaces} that a linear real-valued functional $\phi : X \to \mathbb{R}$ is in $X^*$ if and only if there exists a $C \ge 0$ such that $\phi(x) \le C \Vert x \Vert_X$, i.e. if and only if $\phi$ is bounded from above on the unit ball of $X$.
Moreover, using the Hahn-Banach extension theorem as in \cite[Theorem 2.2.2]{src:Cobzaş:FunctionalAnalysisinAsymmetricNormedSpaces}, one can show that for every $x \in X$ with $\Vert x \Vert_X > 0$ there is a $\phi \in X^*$ with $\Vert \phi \Vert = 1$ such that $\phi(x) = \Vert x \Vert_{X^*}$.
It follows that for every $x \in X$, we have
\begin{equation}
\label{eq:NormByDualFunctionalsIsAttained}
\left\Vert x \right\Vert_X = \max_{\substack{\phi \in X^*\\ \Vert \phi \Vert_{X^*} \le 1}} \phi \left( x \right) \, .
\end{equation}
There is also an obvious interplay with the dual space $\bar{X}^*$ of the conjugate space, namely $X^* = - \bar{X}^*$.
\begin{definition}
Let $X$ be an asymmetrically normed space.
Then every $x \in X$ generates a linear evaluation map $\mathrm{ev}_x : X_s^* \to \mathbb{R}, \phi \mapsto \phi(x)$.
\end{definition}
\begin{lemma}
\label{lem:EvaluationMap}
For every $x$ in an asymmetrically normed space $X$, the evaluation map $\mathrm{ev}_x$ restricted to $X^*$ or $\bar{X}^*$ is continuous.
\end{lemma}
\begin{proof}
Let $\phi \in X^*$ and $(\psi_n)_{n \in \mathbb{N}}$ be a null sequence in $X^*$.
If $\Vert x \Vert_{X} = 0$, we have $\psi_n(x) \le 0$ for all $n \in \mathbb{N}$, since every element in $X^*$ is bounded from above by some non-negative constant times $\Vert \cdot \Vert_{X}$.
Hence, $\limsup_{n \to \infty}  [ \phi( x ) + \psi_n( x ) ] \le \phi( x )$.
If $\Vert x \Vert_{X}  \neq 0$, we obtain
\begin{equation}
\begin{aligned}
\limsup_{n \to \infty} \left[ \phi \left( x \right) + \psi_n \left( x \right) \right]
&=
\phi \left( x \right) + \left\Vert x \right\Vert_{X} \limsup_{n \to \infty} \psi_n \left( \frac{x}{\left\Vert x \right\Vert_{X} } \right) \\
&\le
\phi \left( x \right) + \left\Vert x \right\Vert_{X} \limsup_{n \to \infty}  \left\Vert \psi_n \right\Vert_{X^*}
=
\phi \left( x \right) \, .
\end{aligned}
\end{equation}
Consequently, $\mathrm{ev}_x \upharpoonright X^*$ is upper semicontinuous.
To see the lower semicontinuity, note that by the upper semicontinuity of $\mathrm{ev}_{-x} \upharpoonright X^*$, we obtain
\begin{equation}
- \liminf_{n \to \infty} \left[ \phi \left( x \right) + \psi_n \left( x \right) \right]
=
\limsup_{n \to \infty} \left[ \phi \left( -x \right) + \psi_n \left( -x \right) \right]
\le
\phi \left( -x \right)
=
- \phi \left( x \right)
\end{equation}
and thus
\begin{equation}
\liminf_{n \to \infty} \left[ \phi \left( x \right) + \psi_n \left( x \right) \right]
\ge
\phi \left( x \right) \, .
\end{equation}
By exhcanging the roles of $X$ and $\bar{X}$ the same holds for $\mathrm{ev}_x \upharpoonright \bar{X}^*$.
\end{proof}
With the above \namecref{lem:EvaluationMap} in mind, we shall also simply write $x$ in place of $\mathrm{ev}_x$.
As a final topological notion, we say that a sequence $(x_n)_{n \in \mathbb{N}}$ \textbf{converges weakly} to $x \in X$ whenever
\begin{equation}
\limsup_{n \to \infty} \phi \left( x_n - x \right) \le 0
\end{equation}
for all $\phi \in X^*$ (see \cite[Proposition 2.4.33]{src:Cobzaş:FunctionalAnalysisinAsymmetricNormedSpaces}).
\subsection{Elements of Convex Analysis}
We follow the conventions of \cite{src:Zalinescu:ConvexAnalysisInGeneralVectorSpaces} straightforwardly adapted to the conic setting.
In particular, $\overline{\mathbb{R}} = \{ -\infty \} \cup \mathbb{R} \cup \{ \infty \}$ and a function $f : X \to \overline{\mathbb{R}}$ on a cone (e.g. linear space) $X$ is \textbf{convex} if its \textbf{epigraph}
\begin{equation}
\mathrm{epi}\, f = \left\{ (x, t) \in X \times \mathbb{R} : f \left( x \right) \le t \right\}
\end{equation}
is convex.
Moreover, $f$ is \textbf{proper} if $f(X) \subseteq (-\infty, \infty]$ and $f(X) \neq \{ \infty \}$.
On an asymmetrically normed space $X$ we define the \textbf{convex conjugate} $f^* : \bar{X}^* \to \overline{\mathbb{R}}$ of a function $f : X \to \overline{\mathbb{R}}$ as
\begin{equation}
f^* \left( \phi \right)
=
\sup_{x \in X} \left[ \phi \left( x \right) - f \left( x \right) \right]
\end{equation}
for all $\phi \in \bar{X}^*$.
Obviously, $f^*$ is convex and by \cref{lem:EvaluationMap} it is the supremum over a family of continuous functions and thus lower semicontinuous.
Given a function $g : \bar{X}^* \to \overline{\mathbb{R}}$, we may also consider its conjugate $g^* : X \to \overline{\mathbb{R}}$, that we define as
\begin{equation}
g^* \left( x \right)
=
\sup_{\phi \in \bar{X}^*} \left[ \phi \left( x \right) - g \left( \phi \right) \right]
\end{equation}
for all $x \in X$.
$g^*$ is clearly convex as well.
Moreover, since $\bar{X}^* = - X^*$, every $\phi \in \bar{X}^*$ is lower semicontinuous on $X$ such that $g^*$ is lower semicontinuous as the supremum over a family of lower semicontinuous functions.
It is also obvious that whenever $f_1 \le f_2$ it follows that $f_2^* \le f_1^*$ (regardless of whether $f_1, f_2$ are defined on $X$ or $\bar{X}^*$).
\begin{lemma}
\label{lem:DoubleConjugateIsSmaller}
Let $X$ be an asymmetrically normed space and $f : X \to \overline{\mathbb{R}}$.
Then $f^{**} \le f$.
\end{lemma}
\begin{proof}
Letting $x \in X$, we obtain
\begin{equation}
\sup_{\phi \in \bar{X}^*} \left(
\phi \left( x \right)
-
\sup_{y \in X} \left[
\phi \left( y \right)
-
f \left( y \right)
\right]
\right)
=
\sup_{\phi \in \bar{X}^*} 
\inf_{y \in X} \left[
\phi \left( x \right)
-
\phi \left( y \right)
+
f \left( y \right)
\right]
\le
f \left( x \right) \, .
\end{equation}
\end{proof}
\subsection{Gâteaux and Fréchet Derivates}
We also need notions of Gâteaux and Fréchet derivates on normed cones.
While it may not immediately be clear how to extend the classical definitions in normed spaces to normed cones, there exists prior work in the setting of semilinear spaces that suggest promising definitions \cite{src:Galanis:DiffabilityOnSemilinearSpaces,src:Sadeqi:DiffabilityOnSemilinearSpaces}.
\begin{definition}
Let $f : X \to \overline{\mathbb{R}}$ be a function on a normed cone $X$.
Then $f$ is \textbf{right Gâteaux differentiable} at a point $x \in X$ if $f(x) \in \mathbb{R}$ and there exists some continuous and linear\footnote{\enquote{linear} translates to the usual compatibility with respect to all linear combinations with non-negative scalars. In the vector space setting this is equivalent to the usual meaning of the term.} $L : X^* \to \mathbb{R}$ such that
\begin{equation}
\lim_{ t  \searrow 0 } \frac{1}{t} \left[ f \left( x + t y \right) - f \left( x \right) - t L \left( y \right) \right] = 0
\end{equation}
for all $y \in X$.
In that case $L$ is called the \textbf{Gâteaux derivative} of $f$ at $x$ and is written as $L = Df(x)$.
Moreover $f$ is \textbf{right Fréchet differentiable} at a point $x \in X$ if it is Gâteaux differentiable at $x$ and
\begin{equation}
\lim_{ t  \searrow 0 } \sup_{\Vert y \Vert = t} \frac{1}{t} \left| f \left( x + y \right) - f \left( x \right) - D f \left( x \right) \left( y \right) \right| = 0 \, .
\end{equation}
Then $Df(x)$ is called the \textbf{right Fréchet derivative} of $f$ at $x$.
These derivatives directly generalise the analogous notions on normed spaces and it is clear that the derivative is unique whenever it exists.
\end{definition}
\section{The Fréchet Case}
The asymmetric case becomes an interesting generalisation of \cite[Theorem 3.9.1]{src:Zalinescu:ConvexAnalysisInGeneralVectorSpaces}.
\begin{theorem}
\label{thm:FrechetCase}
Let $(X, \Vert \cdot \Vert_X)$ be an asymmetrically normed space and $f : X \to \overline{\mathbb{R}}$ a proper function, $x \in X$ and $\phi \in \bar{X}^*$.
Setting $g = f - \phi$, consider the following statements:
\begin{enumerate}[(i)]
\item $f$ is lower semicontinuous at $x$ and $f^*$ is right Fréchet differentiable at $\phi$ with $D f^* (\phi) = x$.
Moreover, there exists a sequence $(y_n)_{n \in \mathbb{N}}$ such that
\begin{equation}
\lim_{n \to \infty} g(y_n) = \inf g(X) \qquad\text{and}\qquad \lim_{n \to \infty} y_n = x \text{ in } X, \text{i.e. } \lim_{n \to \infty} \left\Vert y_n - x \right\Vert_X = 0 \, .
\end{equation}
\item $g(x) =  \inf g(X)$ and for every sequence $(y_n)_{n \in \mathbb{N}}$ in $X$ we have
\begin{equation}
 \lim_{n \to \infty} g(y_n) = g(x) \implies \lim_{n \to \infty} y_n = x \text{ in } \bar{X}, \text{i.e. } \lim_{n \to \infty} \left\Vert x - y_n \right\Vert_X = 0 \, .
\end{equation}
\item $f(x) \in \mathbb{R}$ and $f^{**}(x) = f(x)$.
\end{enumerate}
Then $(i) \implies (ii) \implies (iii)$.
Moreover, $(i) \iff (ii)$ whenever $f$ is convex.
\end{theorem}
\begin{remark}
\label{rem:WhyNotLeftDiffability}
The existence of a suitable minimising sequence demanded in $(i)$ is not necessary in the normed setting.
This is due to the fact, that the right differentiability then also implies a left differentiability.
However, if we additionally imposed a suitably generalised left differentiability condition in the asymmetric setting, it would imply that the minimisers in $(ii)$ also converge in $X$ and consequently in $X_s$.
Hence, one would arrive in the normed setting again such that we preferred to keep it this way, even though the hypothesis in $(i)$ might be more difficult to verify.
\end{remark}
The proof is similar to the one in \cite[Theorem 3.9.1]{src:Zalinescu:ConvexAnalysisInGeneralVectorSpaces} which covers the setup of normed vector spaces.
\begin{proof}
$(ii) \implies (iii)$:
We have $g(x) = f(x) - \phi(x)$ and since $f$ is proper, this is not equal to $-\infty$.
Because $x$ minimises $g$, the propriety also implies that $g(x) \neq \infty$.
Hence, $g(x) \in \mathbb{R}$ and thus $f(x) \in \mathbb{R}$.
We also obtain
\begin{equation}
f^* \left( \phi \right)
=
\sup_{y \in X} \left[ \phi \left( y \right) - f \left( y \right) \right]
=
- \inf g \left( X \right)
=
- g \left( x \right)
=
\phi \left( x \right) - f \left( x \right) \, .
\end{equation}
Consequently, $f (x) = \phi(x) - f^*( \phi )$ such that $f(x) = f^{**}(x)$ by \cref{lem:DoubleConjugateIsSmaller}.\\

The above also proves that in case $(ii)$, $f^*(\phi) \in \mathbb{R}$ which is also trivially true in case $(i)$.
Hence, for the remainder of the proof, we may define the function
\begin{equation}
h : X \to \overline{\mathbb{R}}, \qquad y \mapsto
f \left( x + y \right) + f^* \left( \phi \right) - \phi \left( x + y \right)
=
g \left( x + y \right) - \inf g \left( X \right) \, .
\end{equation}
and obtain the conjugate function
\begin{equation}
\begin{aligned}
h^*(\psi)
&=
\sup_{y \in X} \left[ \left( \phi + \psi \right) \left( x + y \right) - f \left( x + y \right) \right]
- f^* \left( \phi \right) - \psi \left( x \right) \\
&=
f^* \left( \phi + \psi \right) - f^* \left( \phi \right) - \psi \left( x \right) \, .
\end{aligned}
\end{equation}

$(i) \implies (ii)$:
By the Fréchet differentiability of $f^*$ at $\phi$, there exists a function $\delta : (0, \infty) \to (0, \infty]$ such that
\begin{equation}
\left\Vert \psi \right\Vert_{\bar{X}^*} \le \delta \left( \epsilon \right)
\implies
\left| h^* \left( \psi \right) \right|
\le
\epsilon \left\Vert \psi \right\Vert_{\bar{X}^*}
\end{equation}
for all $\epsilon > 0$ and $\psi \in \bar{X}^*$.
Let us define the function
\begin{equation}
\alpha : \mathbb{R}_{\ge 0} \to [0, \infty], \qquad t \mapsto t \inf \left \{ \epsilon > 0 : \delta \left( \epsilon \right) \ge t \right \} \, .
\end{equation}
It follows that $| h^* ( \psi )|  \le \alpha ( \Vert \psi \Vert_{\bar{X}^*} )$ for all $\psi \in \bar{X}^*$.
Also, $\alpha( t ) \le \epsilon t$ for all $\epsilon > 0$ with $\delta( \epsilon ) \ge t$. 
Since $\alpha(0) = 0$, the following conjugate function $\alpha^\#$ of $\alpha$ is non-negative with
\begin{equation}
\alpha^{\#} : [0, \infty) \to [0, \infty], \qquad
s \mapsto \sup_{t \ge 0} \left[
t s - \alpha \left( t \right)
\right]
\end{equation}
for all $s \ge 0$.
$\alpha^{\#}$ is obviously monotonically increasing and moreover, for any $s > 0$, we have
\begin{equation}
\label{eq:AlphaHashIsPositiveForPositiveArguments}
\alpha^{\#} \left( s \right)
\ge
\sup_{t \in [0, \delta(s/2)]} \left[
t s - \alpha \left( t \right)
\right]
\ge
\sup_{t \in [0, \delta(s/2)]} \left[
t s - \frac{st}{2}
\right]
=
\frac{s}{2} \delta \left( \frac{s}{2} \right)
>
0 \, .
\end{equation}
Using \cref{eq:NormByDualFunctionalsIsAttained}, it is easy to see that $(\alpha \circ \Vert \cdot \Vert_{\bar{X}^*})^* : X \to \overline{\mathbb{R}}$ is equal to $\alpha^{\#} \circ \Vert \cdot \Vert_{\bar{X}}$.
Since $h^* \le \alpha \circ \Vert \cdot \Vert_{\bar{X}^*}$ it then follows that $\alpha^{\#} ( \Vert y \Vert_{\bar{X}} )\le h^{**} ( y )$ for all $y \in X$.
Moreover,
\begin{equation}
\begin{aligned}
h^{**} \left( y  - x \right)
&=
\sup_{\psi \in \bar{X}^*} \left[
\psi \left( y \right)
-
f^* \left( \phi + \psi \right)
\right]
+
f^* \left( \phi \right) \\
&\le
\sup_{\psi \in \bar{X}^*} \left[
\psi \left( y \right)
-
f^* \left( \psi \right)
\right]
+
f^* \left( \phi \right)
-
\phi \left( y \right) \\
&=
f^{**} \left( y \right)
+
f^* \left( \phi \right)
-
\phi \left( y \right)
\end{aligned}
\end{equation}
where the inequality arises because the inclusion $\phi + \bar{X}^* \subseteq \bar{X}^*$ may be proper.
Hence,
\begin{equation}
f \left( y \right)
\ge
f^{**} \left( y \right)
\ge
h^{**} \left( y - x \right)
-
f^* \left( \phi \right)
+
\phi \left( y \right)
\ge
\alpha^{\#} \left( \left\Vert y - x \right\Vert_{\bar{X}} \right)
-
f^* \left( \phi \right)
+
\phi \left( y \right)
\end{equation}
and thus, finally,
\begin{equation}
g \left( y \right)
\ge
\alpha^{\#} \left( \left\Vert y - x \right\Vert_{\bar{X}} \right)
-
f^* \left( \phi \right)
=
\alpha^{\#} \left( \left\Vert y - x \right\Vert_{\bar{X}} \right)
+
\inf g \left( X \right) \, .
\end{equation}
Let $(y_n)_{n \in \mathbb{N}}$ be a sequence in $X$ with $\lim_{n \to \infty} g( y_n ) = \inf g(X)$.
Then,
\begin{equation}
\inf g(X)
=
\lim_{n \to \infty} g \left( y_n \right)
\ge
\limsup_{n \to \infty} \alpha^{\#} \left( \left\Vert y_n - x \right\Vert_{\bar{X}} \right)
+
\inf g(X)
\end{equation}
such that $\limsup_{n \to \infty} \alpha^{\#} \left( \left\Vert y_n - x \right\Vert_{\bar{X}} \right) = 0$.
By the monotonicity of $\alpha^\#$ and \cref{eq:AlphaHashIsPositiveForPositiveArguments} it follows that $\lim_{n \to \infty} \Vert y_n - x \Vert_{\bar{X}} \to 0$.
Now, pick a sequence $(y_n)$ in $X$ with $\lim_{n \to \infty} \Vert y_n - x \Vert = 0$ and $\lim_{n \to \infty} g(y_n) = \inf g(X)$.
Then we also have $\lim_{n \to \infty} \Vert x - y_n \Vert = 0$ and by the lower semicontinuity of $f$
\begin{equation}
\inf g \left( X \right)
=
\liminf_{n \to \infty} g \left( y_n \right)
\ge
f \left( x \right)
- \phi \left( x \right)
- \limsup_{n \to \infty} \phi \left( y_n - x \right)
\ge
g \left( x \right) \, .
\end{equation}

$(ii) \implies (i)$ if $f$ is convex:

The constant sequence $y_n = x$ obviously satisfies $\lim_{n \to \infty} \Vert y_n - x \Vert = 0$ and $\lim_{n \to \infty} g(y_n) = \inf g(X)$.
Let us define the function
\begin{equation}
\alpha : [0,\infty) \to [0,\infty], \qquad
t \mapsto \inf_{\Vert y \Vert_{\bar{X}} = t} \left[ h \left( y \right) \right] \, .
\end{equation}
Clearly, $\alpha(0) = 0$ and it is obvious that $t > 0 \implies \alpha(t) > 0$.
Now let $0<s<t$ and $y\in X$ such that $\Vert y \Vert_{\bar{X}} = t.$
Then
\begin{equation}
\alpha \left( s \right) 
\leq
h \left( \frac{s}{t} y \right)
\leq
\frac{s}{t} h \left( y \right)
\end{equation}
since $h$ is convex and $h(0) = 0$.
Taking the infimum over all such $y$, we obtain 
\begin{equation}
\frac{\alpha(s)}{s} \leq \frac{\alpha(t)}{t}.
\end{equation}
It follows that $\lim_{s\to 0^+} \alpha^\#(s) / s =0.$ 
To see this we give the proof of \cite[Lemma 3.3.1 (iii)]{src:Zalinescu:ConvexAnalysisInGeneralVectorSpaces}:
Take $\epsilon >0$ and let $0\le s \le \alpha(\epsilon) / \epsilon$.
\begin{equation}
    \begin{aligned}
        \alpha^\#(s) &= \sup_{t\ge 0} \left[ ts -\alpha\left( t \right) \right] \\
                    &= \max \left\{\sup_{0\le t \le \epsilon} \left[ t s -\alpha \left( t \right) \right] , \sup_{t\ge \epsilon} \left[t s - \alpha \left( t \right) \right] \right\} \\
                    &\le \max \left\{ s \epsilon , \sup_{t\ge \epsilon} \left[ t \left( s - \frac{\alpha \left( \epsilon \right)}{\epsilon }\right) \right] \right\} \\
                    &= s \epsilon \,.
    \end{aligned}
\end{equation}
By definition $0 \leq \alpha(\Vert y \Vert_{\bar{X}}) \leq h( y )$.
Therefore,
\begin{equation}
h^* \left( \psi \right)
=
f^* \left( \phi + \psi \right) - f^* \left( \phi \right) - \psi \left( x \right)
\le
\alpha^\# \left( \left\Vert \psi \right\Vert_{\bar{X}^*} \right)
\end{equation}
for all $\psi \in \bar{X}^*$.
Moreover,
\begin{equation}
h^* \left( \psi \right)
\ge
\phi \left( x \right) - f \left( x \right) - f^* \left( \phi \right)
=
0.
\end{equation}
This implies
\begin{equation}
0 \le 
\lim_{\Vert \psi \Vert_{\bar{X}^*} \to 0} \frac{f^* \left( \phi + \psi \right)-f^* \left( \phi \right) - \psi \left( x \right)}{\Vert \psi \Vert} 
\leq
\lim_{t \to 0^+}\frac{\alpha^\#(t)}{t} = 0.
\end{equation}
So $f^*$ is indeed right Fréchet differentiable at $\phi$ with $D f^*(\phi) = x$.
That $f$ is lower semicontinuous at $x$ is clear since $g$ is lower semicontinuous at its minimising point $x$ and $\phi$ is lower semicontinuous on all of $X$.
\end{proof}
\section{The Gâteaux Case}
\begin{theorem}
Let $X$ be an asymmetrically normed space and $f : X \to \overline{\mathbb{R}}$ a proper function, $x \in X$ and $\phi \in \bar{X}^*$.
Setting $g = f - \phi$, consider the following statements:
\begin{enumerate}[(i)]
\item $f$ is weakly lower semicontinuous at $x$ and $f^*$ is right Gâteaux differentiable at $\phi$ with $D f^* (\phi) = x$.
Moreover, there exists a sequence $(y_n)_{n \in \mathbb{N}}$ such that
\begin{equation}
\lim_{n \to \infty} g(y_n) = \inf g(X) \qquad\text{and}\qquad \lim_{n \to \infty} y_n = x \text{ weakly in } X \, .
\end{equation}
\item $g(x) =  \inf g(X)$ and for every sequence $(y_n)_{n \in \mathbb{N}}$ in $X$ we have
\begin{equation}
 \lim_{n \to \infty} g(y_n) = g(x) \implies \lim_{n \to \infty} y_n = x \text{ weakly in } \bar{X} \, .
\end{equation}
\item $f(x) \in \mathbb{R}$ and $f^{**}(x) = f(x)$.
\end{enumerate}
Then $(i) \implies (ii) \implies (iii)$.
Moreover, $(i) \iff (ii)$ whenever $f$ is convex.
\end{theorem}
\begin{remark}
As outlined in \cref{rem:WhyNotLeftDiffability}, the existence of a suitable minimising sequence demanded in $(i)$ is not necessary in the normed setting and could be eliminated by additionally imposing a suitable left differentiability as well.
This would however take us to the normed setting again which seemed undesirable to the authors.
\end{remark}
Again, the proof is similar to the one in \cite[Theorem 3.9.1]{src:Zalinescu:ConvexAnalysisInGeneralVectorSpaces} in the setup of normed vector spaces.
\begin{proof}
The proof that $(ii) \implies (iii)$ is just as in the last section.
Moreover, in both cases $(i)$ and $(ii)$ we may define the function
\begin{equation}
h : X \to \overline{\mathbb{R}}, \qquad y \mapsto
f \left( x + y \right) + f^* \left( \phi \right) - \phi \left( x + y \right)
=
g \left( x + y \right) - \inf g \left( X \right)
\end{equation}
with the conjugate function
\begin{equation}
h^*(\psi)
=
f^* \left( \phi + \psi \right) - f^* \left( \phi \right) - \psi \left( x \right) \, .
\end{equation}
$(i) \implies (ii)$:
First, recall that $f^*(\phi) = - \inf g(X)$.
Let $(y_n)_{n \in \mathbb{N}}$ be a sequence in $X$ with $\lim_{n \to \infty} g(y_n) = \inf g(X)$.
If $y_n$ does not converge weakly to $x$ in $\bar{X}$, there exists a $\psi \in \bar{X}^*$ and a subsequence $(z_n)_{n \in \mathbb{N}}$ of $(y_n)$ such that $\psi(z_n - x) \ge 1$ for all $n \in \mathbb{N}$.
Hence, for all $t > 0$ and $n \in \mathbb{N}$,
\begin{equation}
\begin{aligned}
f^* \left( \phi +t \psi \right) - f^* \left( \phi \right) 
&\ge
\phi \left( z_n \right) + t \psi \left( z_n \right) - f \left( z_n \right) + \inf g \left( X \right) \\
&\ge
t \psi \left( x \right) + t - g \left( z_n \right) + \inf g \left( X \right) \, .
\end{aligned}
\end{equation}
By taking the limit $n \to \infty$, it follows that $f^*(\phi+t\psi) -f^*(\phi)\geq t [ \psi(x) + 1 ]$.
This is a contradiction since the right Gâteaux derivative at $\phi$ in the direction of $\psi$ is equal to $x$ by assumption.
Consequently, $y_n$ converges weakly to $x$ in $\bar{X}$.
Now, pick a sequence $(y_n)_{n \in \mathbb{N}}$ with $\lim_{n \to \infty} g (y_n ) = \inf g(X)$ that converges weakly to $x$.
Then by the above considerations, it converges weakly in $\bar{X}$ as well such that
\begin{equation}
\inf g \left( X \right)
=
\liminf_{n \to \infty} g \left( y_n \right)
\ge
f \left( x \right)
- \phi \left( x \right)
- \limsup_{n \to \infty} \phi \left( y_n - x \right)
\ge
g \left( x \right) \, .
\end{equation}

$(ii) \implies (i)$ if $f$ is convex:
Clearly, the constant sequence $y_n = x$ converges weakly to $x$ in $X$ and satisfies $\lim_{n \to \infty} g(y_n) = g(x) = \inf g(X)$.
By assumption, $x$ minimises $g$ such that $f^*(\phi) = - g(x)$.
Hence, fixing any $\psi \in \bar{X}^*$, we have
\begin{equation}
f^* \left( \phi + t \psi \right) - f^* \left( \phi \right) 
\ge
\phi \left( x \right) + t \psi \left( x \right) - f \left( x \right) + f \left( x \right) - \phi \left( x \right) \\
=
t \psi \left( x \right)
\end{equation} 
for every $t>0$.
The function $F : \mathbb{R} \to \overline{\mathbb{R}}, t \mapsto f^* \left( \phi + t \psi \right)$ is clearly convex and proper with $F(0) \in \mathbb{R}$.
Consequently, the following limit exists (see e.g. \cite[Theorem 2.1.5(iii)]{src:Zalinescu:ConvexAnalysisInGeneralVectorSpaces}) and
\begin{equation}
\lim_{t \searrow 0}\frac{1}{t} \left[ f^* \left( \phi + t \psi \right) - f^* \left( \phi \right) \right] \geq \psi(x) \, .
\end{equation}
Assume that the strict inequality holds.
Then there exists a $\mu >0$ such that
\begin{equation}
f^* \left( \phi + t \psi \right) - f^*(\phi)
>
t \left( \psi \left( x \right ) + \mu \right)
\end{equation}
for small enough $t > 0$.
Hence, for sufficiently large $n \in \mathbb{N}$,
\begin{equation}
f^*\left( \phi + \frac{1}{n} \psi \right) - f^* \left( \phi \right) - \frac{1}{n} \psi \left( x \right) 
>
\frac{1}{n} \mu \, .
\end{equation}
Consequently, there exists a sequence $(y_n)$ in $X$ and $N \in \mathbb{N}$ with
\begin{equation}
\begin{aligned}
f^* \left( \phi+\frac{\psi}{n} \right) - f^* \left( \phi \right) - \frac{\psi \left( x \right)}{n}
&\ge
\left[ \left( \phi+\frac{\psi}{n} \right) \left( x + y_n \right) - f \left( x + y_n \right) \right]
+ g \left( x \right)
- \frac{\psi \left( x \right)}{n} \\
&=
\frac{1}{n} \psi \left( y_n \right) - g \left( x + y_n \right) + g \left( x \right)
\ge
\frac{\mu}{n}.
\end{aligned}
\end{equation}
for all $n \in \mathbb{N}_{\ge N}$.
But then also
\begin{equation}
\label{eq:BoundPsiYnOverNFromBelow}
\frac{1}{n} \psi \left( y_n \right)
>
\frac{1}{n} \psi \left( y_n \right) -\frac{\mu}{n}
\ge
g \left ( x + y_n \right) - g \left( x \right)
\ge
0 \, .
\end{equation}
This implies that $\psi( y_n ) \ge \mu > 0$ for all $n \in \mathbb{N}_{\ge N}$.
Consequently, $(y_n)_{n \in \mathbb{N}}$ contains no subsequences that converge weakly to zero in $\bar{X}$.
Let us assume that there is a subsequence $(y_{n_k})_{k \in \mathbb{N}}$ for which $(\psi (y_{n_k}))_{k \in \mathbb{N}}$ is bounded from above.
Then by \cref{eq:BoundPsiYnOverNFromBelow} we have
\begin{equation}
0 
\le
g \left( x + y_{n_k} \right) - g \left( x \right) 
\le
\frac{1}{n_k} \psi \left( y_{n_k} \right) \to 0 \quad (n \to \infty) \, .
\end{equation}
Hence $\lim_{k \to \infty} g( x + y_{n_k} ) = g(x)$ and by assumption $y_{n_k}$ converges weakly to zero in $\bar{X}$.
This contradicts our previous statement that $(y_n)_{n \in \mathbb{N}}$ cannot contain such subsequences.
Therefore, $\lim_{n \to \infty} \psi( y_n ) = \infty$.

Now suppose $\limsup_{n \to \infty} \psi( y_n ) / n < \infty$ and define $t_n =\sqrt{\psi( y_n )}$ as well as $x_n := y_n / t_n$.
Clearly, $\lim_{n \to \infty} \psi( x_n ) = \infty$.
For sufficiently large $n \in \mathbb{N}$, we have $\frac{1}{t_n} < 1$, so that $x + y_n / t_n$ becomes the convex combination
\begin{equation}
x + \frac{y_n}{t_n}
=
\frac{1}{t_n} \left( x + y_n \right) + \frac{t_n -1}{t_n} x \, .
\end{equation}
Since $f$ is convex, $g$ is also convex and we get
\begin{equation}
\begin{aligned}
0
&\le
g \left( x + x_n \right) - g \left( x \right) 
= 
g \left( x + \frac{y_n}{t_n} \right) - g \left( x \right) \\
&\le
\frac{1}{t_n} g \left( x + y_n \right) - \frac{1}{t_n} g \left( x \right)
\le
\frac{\psi \left( y_n \right)}{n t_n} \to 0 \quad (n \to \infty).
\end{aligned}
\end{equation}
This implies that $x_n$ converges weakly to zero in $\bar{X}$ which contradicts $\lim_{n \to \infty} \psi( x_n ) = \infty$.
Therefore, $\limsup_{n \to \infty} \psi( y_n ) / n =\infty$.\\
Fix some increasing and divergent sequence $(n_k)_{k \in \mathbb{N}}$ in $\mathbb{N}$ with $\lim_{n \to \infty} \psi( y_{n_k} ) / n_k = \infty$.
Set $t_k = \psi( y_{n_k} ) / \sqrt{n_k}$ and $x_k= y_{n_k} / t_k$.
As before, $\lim_{k \to \infty} t_k = \infty = \lim_{k \to \infty} \psi( x_k )$ and we obtain the convex combination
\begin{equation}
x + x_k
=
x + \frac{y_{n_k}}{t_k}
=
\frac{1}{t_k} \left( x + y_{n_k} \right) + \frac{t_k -1}{t_k} x
\end{equation}
for large enough $k \in \mathbb{N}$.
Consequently,
\begin{equation}
\begin{aligned}
0
&\le
g \left( x + x_k \right) - g \left( x \right)
=
g \left( x + \frac{y_{n_k}}{t_k} \right) - g \left( x \right) \\
&\le
\frac{1}{t_k} g \left( x + y_{n_k} \right) - \frac{1}{t_k} g \left( x \right)
\le
\frac{\psi \left( y_{n_k} \right)}{n_k t_k}
=
\frac{1}{\sqrt{n_k}}
\to 0 \quad (k \to \infty) \, .
\end{aligned}
\end{equation}
This implies that $x_k$ converges weakly to zero in $\bar{X}$ which contradicts $\lim_{k \to \infty} \psi( x_k ) = \infty$.
It follows that there cannot be a $\mu > 0$ as assumed and, in fact,
\begin{equation}
\lim_{t \searrow 0}\frac{1}{t} \left[ f^* \left( \phi + t \psi \right) - f^* \left( \phi \right) \right] = \psi(x) \, .
\end{equation}
Since $\psi \in \bar{X}^*$ was arbitrary, $f^*$ is right Gâteaux differentiable at $\phi$ with $D f^*(\phi) = x$.
That $f$ is weakly lower semicontinuous at $x$ follows since $g$ is weakly lower semicontinuous at its minimising point $x$ and $\phi$ is weakly lower semicontinuous on all of $X$.
\end{proof}

\section{Example}
Finally, we present two examples where the above theorems apply.
\begin{example}
Let $l^2$ be the Hilbert space of real-valued square-summable sequences and consider the asymmetric norm
\begin{equation}
p \left( x \right)
=
\sqrt{
\sum_{n = 1}^{\infty}
\max \{ 0, x_n \}^2
}
\end{equation}
and set $f(x) = \bar{p}(x)^2$ for all $x \in l^2$.
Then $f$ is clearly convex and proper.
Setting $X = (l^2, p)$, it is easy to see that
\begin{equation}
\bar{X}^*
=
\left\{
x \in l^2 \mid
\forall n \in \mathbb{N} :
x_n < 0
\right\}
\end{equation}
with the standard pairing $\langle\cdot,\cdot\rangle$ given by the $l^2$ inner product.
Moreover, the norm $\Vert \cdot \Vert_{\bar{X}^*}$ is simply equal to $\bar{p}$.
Letting $y \in \bar{X}^*$ be arbitrary, we have
\begin{equation}
f^* \left( y \right)
=
\sup_{x \in l^2}
\left[ \left\langle y, x \right\rangle - \bar{p} \left( x \right)^2 \right]
=
\sup_{r \ge 0}
\sup_{\substack{x \in l^2\\\bar{p}(x) = r}}
\left[ \left\langle y, x \right\rangle - \bar{p} \left( x \right)^2 \right]
=
\sup_{r \ge 0}
\left[ r \bar{p} \left( y \right) - r^2 \right] \\
=
\frac{\bar{p} \left( y \right)^2}{4} \, .
\end{equation}
With $g(x) = \bar{p}(x)^2 - \langle y, x \rangle$, we see that $g$ is minimised by
\begin{equation}
g \left( \frac{y}{2} \right)
=
- \frac{1}{4} \bar{p} \left( y \right)^2 \, .
\end{equation}
At the same time, we have
\begin{equation}
\begin{aligned}
g \left( x \right)
&=
- \frac{1}{4} \bar{p} \left( y \right)^2
+ \sum_{n = 1}^\infty \left[
\frac{ y_n^2 }{4}
- y_n x_n
+ \max \{ 0, -x_n \}^2
\right] \\
&=
- \frac{1}{4} \bar{p} \left( y \right)^2
+ \sum_{n = 1}^\infty
\begin{cases}
\left( \frac{\left| y_n \right|}{2} - \left| x_n \right| \right)^2 & \text{if } x_n < 0 \\
\frac{ y_n^2 }{4}
+ \left| y_n \right| \left| x_n \right| & \text{if } x_n \ge 0
\end{cases} \\
&\ge
- \frac{1}{4} \bar{p} \left( y \right)^2
+ \sum_{n = 1}^\infty
\begin{cases}
\left( \frac{y_n}{2} - x_n \right)^2 & \text{if } x_n < 0 \\
0 & \text{if } x_n \ge 0
\end{cases}
\ge
\bar{p} \left( x - \frac{y}{2} \right)^2
+ \inf g \left( l^2 \right)
\end{aligned}
\end{equation}
for every $x \in l^2$.
Consequently, every minimising sequence converges to $y/2$ in $\bar{X}$ such that \cref{thm:FrechetCase} $(ii)$ is satisfied.
By the convexity of $f$ it follows that $f$ is lower semicontinuous at $y/2$ and that $f^*$ is right Frechet differentiable at $y$.
Moreover, since $y$ was arbitrary, these properties hold on all of $\bar{X}$.
\end{example}
The second example is more elaborate and studies properties of moment-generating functions of a measure in a very general setting.
Due to the asymmetry of moment-generating functions, the asymmetrically normed vector space becomes the natural setting as is demonstrated in the following.
\begin{example}
Let $X$ be a real locally convex space and $\mu$ a probability measure on the Borel sets of $X$ that is scalarly of first order, i.e. $X^* \subseteq L^1(\mu)$ and of mean zero, i.e. $\int_X \phi \mathrm{d} \mu = 0$ for all $\phi \in X^*$.
Let $[\dots]_\mu$ denotes the equivalence class of functions equal $\mu$-almost everywhere and define the vector space
\begin{equation}
\mathcal{K}(\mu)
=
\left\{
\left[ \phi \right]_\mu : \phi \in X^*
\right\} \, ,
\end{equation}
as well as its completion $\mathcal{M}(\mu)$ in the topology of convergence in $\mu$-measure.
Here, it should be noted that unless $\mu$ is \enquote{scalarly centred at zero} \cite[Theorem 2]{src:Chevet:KernelOfCylindricalMeasures}, there will be elements in $\mathcal{M}(\mu)$ that are constant and non-zero $\mu$-almost everywhere.
Now, define the cumulant-generating function $V_\mu : \mathcal{K}(\mu) \to \overline{\mathbb{R}}$ with
\begin{equation}
\label{eq:VMuDefinition}
T \mapsto \ln \int_X \exp \left[ T \right] \mathrm{d} \mu
\end{equation}
and let $\overline{W}_\mu$ denote the greatest lower semicontinuous function on $\mathcal{M}(\mu)$ whose restriction to $\mathcal{K}(\mu)$ is smaller than $V_\mu$, i.e. with $\mathrm{epi}\, \overline{W}_\mu = \mathrm{cl}\, \mathrm{epi}\, V_\mu$.
Finally, define $\mathcal{L}(\mu) = \mathcal{M}(\mu) \cap L^1(\mu)$, let $W_\mu = \overline{W}_\mu \upharpoonright \mathcal{L}(\mu)$ and endow $\mathcal{L}(\mu)$ with the asymmetric norm $p : \mathcal{L}(\mu) \to [0, \infty)$ given by
\begin{equation}
T \mapsto \int_X \max \left\{ -T, 0 \right\} \mathrm{d} \mu \, .
\end{equation}
Then it is easy to see that
\begin{enumerate}
\item $V_\mu$, $W_\mu$ and $\overline{W}_\mu$ coincide on $\mathcal{K}(\mu)$,
\item $T \in \mathrm{dom}\, \overline{W}_\mu \implies \overline{W}_\mu(T) = \ln \int_X \exp[T] \mathrm{d} \mu$,
\item $T \in \mathrm{dom}\, \overline{W}_\mu \implies T \in \mathcal{L}(\mu)$ and $\int_X T \mathrm{d} \mu \ge 0$,
\item $\mathrm{gr}\, V_\mu$ is dense in $\mathrm{gr}\, W_\mu$ equipped with the subspace topology of $(\mathcal{L}(\mu), \bar{p}_\mu) \times \mathbb{R}$,
\item $W_\mu \ge \max \{ \ln \Vert \cdot \Vert_{L^1(\mu)}, 0 \}$.
\end{enumerate}
The reason behind this construction is to obtain a space that is sufficiently \enquote{complete} to have minimisers and has good coercivity properties.
Moreover, $W_\mu^*$ can be determined by taking the supremum over $\mathcal{K}(\mu)$ instead of $\mathcal{L}(\mu)$.

With the $L^1(\mu)$-coercivity at hand, it is useful to consider the vector space $(\mathcal{L}(\mu)_s)^*$ with dual norm denoted by $\Vert \cdot \Vert_s^*$.
Then we obtain the following.
\begin{theorem}
Let $y \in \mathrm{dom}\, W_\mu^*$ such that for some $\epsilon > 0$,
\begin{equation}
\sup\left\{ W_\mu^* \left( y + z \right) \middle| z \in (\mathcal{L}(\mu), \bar{p})^*: \left\Vert z \right\Vert_s^* < \epsilon \right\}
<
\infty \, .
\end{equation}
Then there is a $T \in \mathcal{L}(\mu)$ at which the infimum of $W_\mu - y$ is attained and every minimising sequence $(T_n)_{n \in \mathbb{N}}$ of $W_\mu - y$ converges $\bar{p}$-weakly to $T$.
\end{theorem}
\begin{proof}
From the choice of $y$, we have with $B^{s\ast}_\epsilon(0) = (\Vert \cdot \Vert_s^*)^{-1}([0,\epsilon))$,
\begin{equation}
\sup_{z \in B^{s\ast}_\epsilon(0)}
\sup_{\phi \in \mathcal{L}(\mu)}
\left[
\left( y + z \right) \left( \phi \right)
-
W_\mu \left( \phi \right)
\right]
=
\sup_{\phi \in \mathcal{L}(\mu)}
\left[
\epsilon \left\Vert \phi \right\Vert_{L^1(\mu)}
-
\left( W_\mu - y \right) \left( \phi \right)
\right]
<
\infty \, .
\end{equation}
Consequently, $W_\mu - y$ is $L^1(\mu)$-norm coercive.
Hence, the desired infimum is finite and any minimising sequence $(T_n)_{n \in \mathbb{N}}$ is $L^1(\mu)$-bounded.
Hence, by Komlós' theorem, for every subsequence of $(T_n)_{n \in \mathbb{N}}$, there is some $S \in L^1(\mu)$ and a further subsequence $(S_n)_{n \in \mathbb{N}}$ such that
\begin{equation}
R_n := \frac{1}{n} \sum_{m = 1}^n S_m \to T \text{ $\mu$-almost everywhere}.
\end{equation}
Since $W_\mu$ is convex, $(R_n)_{n \in \mathbb{N}}$ is a minimising sequence as well and because $\overline{W}_\mu$ is lower semicontinuous,
\begin{equation}
\liminf_{n \to \infty} W_\mu \left( R_n \right)
\ge
\overline{W}_\mu \left( S \right) \, .
\end{equation}
However, then $S \in \mathcal{L}(\mu)$ and since $y$ is upper-semicontinuous with respect to convergence in measure,
\begin{equation}
\lim_{n \to \infty} \left[ W_\mu \left( R_n \right) - y \left( R_n \right) \right]
=
W_\mu \left( S \right) - y \left( S \right)
=
\ln \int_X \exp \left[ S \right] \mathrm{d} \mu - y \left( S \right) \, .
\end{equation}
Consequently, $\max\{ R_n, 0 \}$ is uniformly $\mu$-integrable and since $\max\{ a + b, 0 \} \le \max\{ a, 0 \} + \max\{ b, 0 \}$ for all $a, b \in \mathbb{R}$, we have
\begin{equation}
\max \left\{ R_n - S, 0 \right\}
\le
\max\{ R_n, 0 \}
+
\max\{ -S, 0 \} \, .
\end{equation}
Since $S \in L^1(\mu)$, $\max \left\{ R_n - S, 0 \right\}$ is uniformly $\mu$-integrable and by Vitali's convergence theorem
\begin{equation}
\lim_{n \to \infty} \bar{p} \left( R_n - S \right)
=
\lim_{n \to \infty} \int_X \max \left\{ R_n - S, 0 \right\} \mathrm{d} \mu
=
0 \, .
\end{equation}
Let $M \subseteq \mathcal{L}(\mu)$ denote the set of minimisers of $W_\mu - y$ and let $S' \in M$.
Then, the function
\begin{equation}
f : [0,1] \to \mathbb{R}, t \mapsto W_\mu \left( \left[ 1 - t \right] S + t S' \right)
=
\ln \int_X \exp \left[ \left( 1 - t \right) S + t S' \right] \mathrm{d} \mu \, .
\end{equation}
is twice differentiable on the open interval $(0,1)$ such that for some $t \in (0,1)$,
\begin{equation}
W_\mu \left( S' \right) - y \left( S' \right)
=
W_\mu \left( S \right) - y \left( S \right)
+
\frac{1}{2} f'' \left( t \right) \, ,
\end{equation}
i.e. $f'' ( t ) = 0$.
Writing $U = (1 - t) S + t S'$ for brevity, we obtain
\begin{equation}
f'' \left( t \right)
=
\frac{\int_X \left( S' - S \right)^2 \exp \left[ U \right] \mathrm{d} \mu}{\int_X \exp \left[ U \right] \mathrm{d} \mu}
- \left( \frac{\int_X \left( S' - S \right) \exp \left[ U \right] \mathrm{d} \mu}{\int_X \exp \left[ U \right] \mathrm{d} \mu} \right)^2
=
0 \, .
\end{equation}
Since this amounts to Hölder's inequality with the functions $S' - S$ and $1$ becoming an equality, we conclude that that $S'$ and $S$ differ by a constant $\mu$-almost everywhere.
Hence, by the convexity of $W_\mu - y$, there exists an interval $I \subseteq \mathbb{R}$ such that
\begin{equation}
M = \left\{ S + c : c \in I \right\} \, .
\end{equation}
Since $W_\mu - y$ is $L^1(\mu)$-coercive, $I$ is bounded.
Moreover, by the lower semicontinuity of $W_\mu - y$, it is immediate that $0 \le \sup I \in I$.
Hence, setting $T = S + \sup I$, we have
\begin{equation}
\lim_{n \to \infty} \int_X \max \left\{ R_n - T, 0 \right\} \mathrm{d} \mu
=
0 \, .
\end{equation}
Consequently, $(T_n)_{n \in \mathbb{N}}$ converges $\bar{p}$-weakly to $T$.
\end{proof}
The striking conclusion is that $W_\mu^*$ is right-Gâteaux differentiable at $y$ with derivative equal to $T$ even though the underlying space $\mathcal{L}(\mu)$ is not even $T_1$ in general.
This generalises a result in \cite{src:Ziebell:RigorousFRG} where it was found that for certain measures $\nu$ being absolutely continuous to Gaußian measures, the ordinary vector space setting suffices and one obtains strong differentiability properties of the corresponding functions $W_\nu^*$.
\end{example}
\printbibliography
\end{document}